\documentclass[11pt,draft,reqno]{amsart}

\usepackage{amssymb} 
\usepackage{mathrsfs} 
\usepackage{stmaryrd} 
\usepackage[dvipsnames]{color}

\DeclareMathAlphabet{\mathpzc}{OT1}{pzc}{m}{it} 

\usepackage{amssymb} 
\usepackage{mathrsfs} 
\usepackage{amsthm}
\usepackage{amsmath}
\usepackage{latexsym}
\usepackage{amsfonts,color}

 \newtheorem{Thm}{Theorem}[section]
 \newtheorem{Cor}[Thm]{Corollary}
 \newtheorem{Lem}[Thm]{Lemma}
 \newtheorem{Prop}[Thm]{Proposition}
 \theoremstyle{definition}
 \newtheorem{Def}[Thm]{Definition}
 \theoremstyle{remark}
 \newtheorem{Rem}[Thm]{Remark}
 \newtheorem*{ex}{Example}
 \numberwithin{equation}{section}

\theoremstyle{definition}

\theoremstyle{definition} 

\newcommand\funzione{\rightarrow}

\newcommand\enne{\mathbb{N}} 
\newcommand\erre{\mathbb{R}} 
\newcommand\ci{\mathbb{C}} 
\newcommand\alg{\mathbb{A}} 
\newcommand\quat{\mathbb{H}} 
\newcommand\oct{\mathbb{O}} 

\DeclareMathOperator{\re}{Re} 
\DeclareMathOperator{\im}{Im} 

\newcommand{\su}{\mathbb{S}}

\newcommand\zetabar{\bar{\zeta}}
\newcommand\vbar{\bar{v}}
\newcommand\zbar{\bar{z}}

\DeclareMathOperator{\de}{d \! \hspace{0.2ex}} 

\providecommand{\clint}[1]{\hspace{0.045ex}[#1]} 

\newcommand{\lra}{\longrightarrow}
\newcommand{\RR}{\mathbb{R}}
\newcommand{\CC}{\mathbb{C}}
\newcommand{\OO}{\Omega}
\newcommand{\mr}{\mathrm}

\newcommand{\mscr}{\mathscr}

\newcommand{\q}{\mathbb{H}}
\newcommand{\oc}{\mathbb{O}}
\newcommand{\mc}{\mathcal}
\newcommand{\ui}{\mathbf{i}} 
\newcommand{\1}{\mr{1}}
\newcommand{\II}{\mc{I}}

\newcommand{\sto}{\mr{u}}
\newcommand{\Sto}{\mr{U}}
\newcommand{\stx}{\mscr{S}}

\begin{document}

%
%
%
%
%
%
%
%
%

\title[Cauchy integral formula]{Noncommutative Cauchy integral formula}

\author{Riccardo Ghiloni}

\address{%
        Dipartimento di Matematica\\ 
        Universit\`a di Trento\\
        Via Sommarive 14\\ 
        38123 Trento\\ 
        Italy }

\email{ghiloni@science.unitn.it}

\thanks{R.\ Ghiloni and A.\ Perotti are partially supported by FIRB 2012 ``Differential Geometry and Geometric Function Theory", MIUR Project ``Propriet\`a geometriche delle variet\`a reali e com\-ples\-se" and GNSAGA of INdAM.
V.\ Recupero is a member of GNAMPA of INdAM}
\author{Alessandro Perotti}

\address{%
        Dipartimento di Matematica\\ 
        Universit\`a di Trento\\
        Via Sommarive 14\\ 
        38123 Trento\\ 
        Italy }

\email{perotti@science.unitn.it}

\author{Vincenzo Recupero}

\address{%
        Dipartimento di Scienze Matematiche\\ 
        Politecnico di Torino\\
        Corso Duca degli Abruzzi 24\\ 
        10129 Torino\\ 
        Italy}
\email{vincenzo.recupero@polito.it}



\date{}

\begin{abstract}
The aim of this paper is to provide and prove the most general Cauchy integral formula for slice regular functions and for $C^1$ functions on a real alternative *-algebra. Slice regular functions represent a generalization of the classical concept of holomorphic {function} of a complex variable in the noncommutative and nonassociative settings. As an application, we obtain two kinds of local series expansion for slice regular functions.
\end{abstract}

\maketitle




\section{Introduction}

One of the main tasks in noncommutative complex analysis is the determination of the class of functions admitting a local power series expansion at every point of their domain of definition. 

Let $\alg$ denote the noncommutative structure we are working with: it may be, for instance,  the skew field {$\quat$} of quaternions, the nonassociative {division algebra} $\oct$ of octonions, the Clifford algebras $\erre_{p,q}$, or any real alternative *-algebra. The noncommutative setting requires a distinction between polynomials with left and right coefficients in $\alg$. If we consider, for instance, coefficients on the right of the indeterminate $x$ (the left case yields an analogous theory), then it is well known that the proper way to perform the multiplication consists in imposing commutativity of $x$ with the coefficients (cf.
\cite{Lam91}). Thus if $p(x) = \sum_n x^n c_n$ and $q(x) = \sum_n x^n d_n$, then their product is defined by
\begin{equation}\label{intro:p*q}
  (p*q)(x) := \sum_{n}x^n\bigg(\sum_{k+h=n} c_k d_h\bigg).
\end{equation}
Note that this product is different from the pointwise product of $p$ and $q$, even if one of the two polynomials is constant: indeed if $p(x) = c_0$ the pointwise product is $p(x)q(x) = \sum_n c_0(x^n d_n)$, while $(p*q)(x) = \sum_n x^n(c_0d_n)$. If instead $q(x) = d_0$ and $p(x) = \sum_n x^n c_n$, then $p(x)q(x) = \sum_n (x^nc_n)d_0$ and 
$(p*q)(x) = \sum_n x^n(c_n d_0)$. 

The problem of the power series representation was solved in the quaternionic case in \cite{GenSto12}: the class of functions admitting a power series expansion is given by the \emph{slice regular} functions. The theory of slice regularity on the quaternionic space was introduced in \cite{GenStr06, GenStr07} (see also \cite{GenStoStr13}), and then it was {extended} to Clifford algebras {and octonions in} \cite{ColSabStr09, GeStRocky}, and to {any real alternative} *-algebras in \cite{GhiPer11pr, GhiPer11}.

The notion of slice regular function generalizes the classical concept of holomorphic function of a complex variable. Let us briefly describe this notion in the simpler case {in which} $\alg$ is $\quat$ or $\oct$. Let $\su$ be the subset of square roots of $-1$ and, for each $J \in \su$, let $\CC_J$ be the plane generated by $1$ and $J$. Observe that each $\CC_J$ is a copy of the complex plane. 
The quaternions and the octonions have a ``slice complex'' nature, described by the following two properties: 
$\alg=\bigcup_{J \in \su}\CC_J$ and $\CC_J\cap \CC_K=\RR$ for every $J,K \in \su$ with $J \neq \pm K$. Let $D$ be an open subset of $\CC$ invariant under complex conjugation and let 
$\OO_D=\bigcup_{J \in \su}D_J$, where 
$D_J:=\{\rho+\sigma J \in \CC_J \, : \, \rho,\sigma \in \RR, \rho+\sigma i \in D\}$. A function $f:\OO_D \lra \alg$ of class $C^1$ is called \emph{slice regular} if, for every $J \in \su$, its restriction $f_J$ to $D_J$ is holomorphic with respect to the complex structures on $D_J$ and {on} $\alg$ defined by the left multiplication by $J$, i.e. if $\partial f_J/\partial \rho+J \, \partial f_J/\partial \sigma=0$ on 
$D_J$. The precise definition of slice regular functions in the most general setting of {real alternative} *-algebras is recalled in Section
\ref{S:Preliminaries} below.

One of the main achievements of the theory of slice regular functions is a Cauchy-type integral formula (see 
\cite{GenStr07, ColGenSab10, GhiPer11, ColSabStr11Mich, ColSab11JMAA}), which has many consequences also in noncommutative functional analysis (cf. \cite{ColSabStr08, ColSabStr11, ColSab11, GhiMorPer13, GhiRec14}).
It generalizes two classical formulas, the Cauchy integral formula for holomorphic functions and its extension to $C^1$ functions, sometimes called Cauchy--Pompeiu formula. Let us show it again in the case $\alg = \quat$ or $\alg = \oct$. If $D$ is bounded, its boundary is piecewise of class $C^1$, and $f$ is of class $C^1$
on the closure of $\OO_D$ in $\alg$, then for every $J \in \su$ it holds:
\begin{equation}\label{intro:cauchy formula}
  f(x) = 
  \frac{1}{2\pi} \int_{\partial D_J} C_y(x) \, J^{-1} \de y \ f(y) -
  \frac{1}{2\pi} \int_{D_J} C_y(x) \, J^{-1} \de \overline y \wedge \de y\ \frac{\partial f}{\partial \overline{y}}(y)
  \quad \forall x \in D_J,
\end{equation}
where the function $C_y$ denotes the \emph{(noncommutative) Cauchy kernel} defined by
\[
C_y(x) := (x^2 - 2\re(y) x + |y|^2)^{-1}(\overline{y}-x).
\]
For $\alg=\q$, the associativity allows to prove that formula \eqref{intro:cauchy formula} holds for all $x\in\Omega_D$.
The two integrals in \eqref{intro:cauchy formula} are defined in a natural way:
\[
  \int_{\partial D_J} C_y(x) \, J^{-1} \de y \ f(y) :=
  \int_0^1 C_{\alpha(t)}(x) \, J^{-1} \alpha'(t) \ f(\alpha(t)) \de t 
\]
and
\[
  \int_{D_J} C_y(x) \, J^{-1} \de \overline{y} \wedge \de y\ \frac{\partial f}{\partial \overline{y}}(y) := 
  2\int_{D_J} C_{\rho + \sigma J}(x) \,  \frac{\partial f}{\partial \overline{y}}(\rho + \sigma J) \de \rho \de \sigma,
\]
$\alpha : [0,1] \funzione \ci_J$ being a Jordan curve parametrizing $\partial D_J$, and $(\rho,\sigma)$ being the real coordinates in 
$\ci_J$. As usual $\partial f/\partial \overline{y} := \frac{1}{2}(\partial f_J/\partial \rho + J\partial f_J/\partial \sigma)$ and the fact that the differential $\de y$ appears on the left of $f(y)$ depends on the noncommutativity of $\alg$. Notice that, if $x$ and $y$  belong to the same $\CC_J$, and {hence} commute, then it turns out that $C_y(x)=(y-x)^{-1}$ and we find again the form of the classical Cauchy {formula} for holomorphic functions.

A drawback of formula \eqref{intro:cauchy formula} is that it is not a representation formula: indeed in the nonassociative case it holds  for $x \in D_J \subseteq \ci_J$ and not {on} the whole domain $\Omega_D$. An elementary example is the one in which $A=\oc$, $D$ is the open disk of $\CC$ centered at the origin with radius $2$,  $f(x):=xi$ and $J:=(i+j)/\sqrt{2}$, where $\{1,i,j,ij,k,ik,jk,k(ij)\}$ is the canonical basis of $\oc$. In fact, in this case, $k \in \OO_D$ and $f(k)=ki$, but the right hand side of formula \eqref{intro:cauchy formula} gives a different value in $x=k$, namely $(ki+kj)/2$.

The aim of the present paper is to find  a Cauchy integral formula, proved in Theorem~\ref{T:Cauchy formula} for general real alternative $^*$-algebras $\alg$, allowing to represent the values $f(x)$ when $x$ belongs to the whole domain $\Omega_D$ of $f$. In order to do this, we exploit the notion of \emph{slice product} between two slice regular functions $f$ and $g$, which is recalled in Definition \ref{D:slice product} below and will be denoted simply by $f \cdot g$. This product is the natural generalization to functions of the product {\eqref{intro:p*q}} of polynomials and allows us to provide the following Cauchy integral representation formula: 
\begin{align}\label{intro:general cauchy formula}
  f(x) 
  & = \frac{1}{2\pi} \int_{\partial D_J} \left[ C_y \cdot  \left( J^{-1} \de y \ f(y) \right) \right](x) -
        \frac{1}{2\pi} \int_{D_J} 
          \Big[ C_y \cdot  \Big( J^{-1} \de \overline y \wedge \de y\ \frac{\partial f}{\partial \overline{y}}(y) \Big) \Big](x)  
\end{align}
holding {for every $J \in \su$ and} \emph{for every $x \in \Omega_D$}, where the parentheses are omitted in the term $J^{-1} \de y \, f(y)$, because this product is proved to be associative (cf.~Remark~\ref{remark31}(iii)). Observe that with respect to $x$ the slice product in the integrand function of \eqref{intro:general cauchy formula} is computed in the variable $x$ between the function $C_y$ and the constant functions $J^{-1} \de y \, f(y)$ and $ J^{-1} \de \overline y \wedge \de y\ \frac{\partial f}{\partial \overline{y}}(y)$, $y$ being the integration variable. Therefore, denoting by $\cdot_x$ the slice product performed with respect to the variable $x$, we {can rewrite formula \eqref{intro:general cauchy formula} in the following more explicit way:}
\begin{multline}
  f(x)
   = \frac{1}{2\pi} \int_{\partial D_J} C_y(x) \cdot_x \, \left( J^{-1} \de y \ f(y) \right) -
         \frac{1}{2\pi} \int_{D_J} C_y(x) \cdot_x \, \Big( J^{-1} \de \overline y \wedge \de y\ \frac{\partial f}{\partial \overline{y}}(y) \Big)  
         \notag \\
   := \frac{1}{2\pi} \int_0^1 \left[ C_{\alpha(t)} \cdot \, \left( J^{-1} \alpha'(t) \ f(\alpha(t)) \right) \right](x) \de t   \\  
         \hfill-\frac{1}{\pi} \int_{D_J} \Big[ C_{\rho + \sigma J} \cdot \, \frac{\partial f}{\partial \overline{y}}{(\rho + \sigma J)} \Big](x) \, {\de\rho \de\sigma}\notag
\end{multline}
for each $x \in \Omega_D$. 
The noncommutative Cauchy kernel $C_y(x)$ is the inverse of the slice regular function $y-x$ with respect to the slice product: $C_y(x) \cdot_x (y-x)=(y-x) \cdot_x C_y(x)=1$.

The keypoint here is that we are considering the set of slice functions as an algebra by using the slice product instead of the pointwise product: whence formulas \eqref{intro:general cauchy formula} and \eqref{intro:cauchy formula} have extremely different natures. We refer the reader to \cite{GhPeSt} for more details concerning the algebra of slice functions.

The new Cauchy integral formula \eqref{intro:general cauchy formula}
allows us to obtain the series expansion at $x_0$ of a slice regular function $f$ with respect to slice powers 
$(x-x_0)^{\cdot n}$ or to spherical polynomials $\mathscr{S}_{x_0,n}(x)$ by using the classical method for complex holomorphic functions (cf.~Propositions~\ref{prop4_1} and \ref{prop4_2} for details). In the associative framework this result was achieved by a different method in \cite{GhiPer14}, where a proof is also sketched for the nonassociative case.

The proof of the Cauchy formula {\eqref{intro:general cauchy formula}} given in Section \ref{Cauchy} is new, also in the associative case.


\section{Preliminaries}\label{S:Preliminaries}


\subsection{Real alternative *-algebras}
Let us assume that
\begin{equation}\label{A real alternative}
  \text{$\alg$ is a finite dimensional real \emph{alternative algebra} with unit},
\end{equation}  
i.e. $\alg$ is a finite dimensional real algebra with unit $1_\alg$ such that the mapping 
\begin{equation}\label{A alternative}
 (x,y,z) \longmapsto (xy)z - x(yz) \quad \text{is alternating}.\notag
\end{equation}
Note that we are not assuming that $\alg$ is associative, but the following theorem holds (cf.\ \cite{Sch66}):

\begin{Thm}[Artin]\label{artin}
The subalgebra generated by any two elements of an alternative algebra is associative.
\end{Thm}

Here we assume that the real dimension of $\alg$ is strictly greater than $1$,
so that $\alg \neq \{0\}$, i.e. $1_\alg \neq 0$.
A consequence of the bilinearity of the product in $\alg$ is the equality
\begin{equation}\label{r(ab) = (ra)b = a(rb)}
  r(xy) = (rx)y = x(ry) \qquad \forall r \in \erre, \quad \forall x, y \in \alg.
\end{equation}
Therefore if we identify $\erre$ with the subalgebra generated by $1_\alg$, then the notation $rx$ is not ambiguous if $r \in \erre$ and $x \in \alg$. Notice that 
\begin{equation}\label{ra = ar}
  rx = xr \qquad \forall r \in \erre, \quad \forall x \in \alg.
\end{equation}
We also assume that $\alg$ is a \emph{*-algebra}, that is
\begin{equation}\label{A *-algebra}
  \text{$\alg$ is endowed with a \emph{*-involution} $\alg \funzione \alg : x \longmapsto x^c$},
\end{equation} 
i.e. a real linear mapping such that
\begin{align*}
   (x^c)^c &= x \qquad\ \forall x \in \alg, \\  \notag
   (xy)^c &= y^c x^c \quad \forall x, y \in \alg, \\  \notag
   r^c &= r \qquad\ \forall r \in \erre.  \notag
\end{align*}
Summarizing the previous assumptions \eqref{A real alternative} and \eqref{A *-algebra}, we say that
\begin{equation}\label{A alternative *-algebra}
  \text{$\alg$ is a \emph{finite dimensional real alternative *-algebra with unity}}.
\end{equation}
We will assume \eqref{A alternative *-algebra} in the remainder of the paper and we will endow $\alg$ with the topology induced by any norm on it as a real vector space.


\begin{Def}
The \emph{trace} $t(x)$ and the \emph{squared norm} $n(x)$ of any $x \in \alg$ are defined as follows
\begin{equation}\label{trace and squared norm}
  t(x) := x + x^c, \quad n(x) := x x^c, \qquad x \in \alg.  \notag
\end{equation}
Moreover, we define $Q_\alg$, the \emph{quadratic cone of $\alg$}, and the set $\su_\alg$ of \emph{square roots of $-1$} by:
\begin{equation}
  Q_\alg := \erre \cup \{x \in \alg\ :\ t(x) \in \erre,\ n(x) \in \erre,\ t(x)^2 - 4n(x) < 0\},  \notag
\end{equation}
\begin{equation}
  \su_\alg := \{J \in Q_\alg \, : \, J^2 = -1\}.  \notag
\end{equation}
For each $J \in \su_\alg$, we denote by $\ci_J :=\langle 1, J \rangle$ the subalgebra of $\alg$ generated by $J$. Finally, the 
\emph{real part $\re(x)$} and the \emph{imaginary part $\im(x)$} of an element $x$ of $Q_\alg$ are given by
\begin{equation}
  \re(x) := (x+x^c)/2, \quad \im(x) := (x-x^c)/2, \qquad x \in Q_\alg.  \notag
\end{equation}
\end{Def}
When $x\in Q_\alg\smallsetminus\erre$, then $\im(x)=sJ$, with $s=\sqrt{n(\im(x))}\in\erre^+$, $J\in\su_\alg$. Note that $\im(x)$  includes also the imaginary unit $J$ of $x$.

Since $\alg$ is assumed to be alternative, one can prove (cf. \cite[Proposition 3]{GhiPer11}) that the quadratic cone $Q_\alg$ has the following two properties, which describe its ``slice complex'' nature:
\begin{align}
  & Q_\alg = \bigcup_{J \in \su_\alg} \ci_J, \label{cone = union of planes} \\
  & \ci_J \cap \ci_K = \erre \qquad \forall J, K \in \mathbb{S}_\alg, \ J \neq \pm K. \label{eq:intersection}
\end{align}
Two simple consequences of \eqref{cone = union of planes} are the following:
\begin{align}
 & \exists x^{-1}=n(x)^{-1} x^c \qquad \forall x \in Q_\alg \smallsetminus \{0\}, \label{inverse of x} \\
 & x^n \in Q_\alg \qquad \forall x \in Q_\alg, \quad \forall n \in \enne.  \notag
\end{align}
Observe that if $x = r+sJ \in \ci_J$, with $r, s \in \erre$, $J \in \su_\alg$, then $x^c = r - s J$. 




\subsection{Slice functions and their slice product}

Let us recall that the complexification of $\alg$ is the real vector space given by the {tensor} product
\begin{equation}
  \alg_\ci := \alg \otimes _\erre \ci \simeq \alg^2,  \notag
\end{equation} 
that can be described by setting $\1 := (1, 0) \in \alg^2$ and $\ui := (0,1) \in \alg^2$, so that every $v = (x,y) \in \alg^2$ can be uniquely written in the form $v = x1 + y\ui = x + y \ui$, and $\ui$ is an \emph{imaginary unit}. Thus the sum in $\alg_\ci$ reads 
$(x+y\ui) + (x'+y'\ui) = (x+x') + (y+y')\ui$
and the product defined by
\begin{equation}\label{product in A_C}
  (x+y\ui)(x'+y'\ui) := (xx'-yy') + (xy'+yx')\ui,  \notag
\end{equation}
makes $\alg_\ci$ a complex alternative algebra as well, therefore $\alg_\ci = \alg + \alg \ui = \{x+y\ui\ :\ x, y \in \alg\}$ and 
$\ui^2= -1$. The \emph{complex conjugation} of $v = x+y\ui \in \alg_\ci$ is defined by $\vbar := x-y\ui$.  

We are now in position to recall the notion of slice functions. Let $D$ be a subset of $\CC$, invariant under the complex conjugation $z=r+si \longmapsto \overline{z}=r-si$, $r, s \in \erre$. Define
\begin{equation}\label{Omega D}
\OO_D:=\{r + s J \in Q_\alg\, : \, r, s \in \RR, J \in \su_\alg, r+s i \in D\}.  \notag
\end{equation}
A subset of $Q_\alg$ is said to be \emph{circular} if it is equal to $\OO_D$ for some set $D$ as above. 

Suppose now that $D$ is open in $\CC$, not necessarily connected. Thanks to \eqref{cone = union of planes} and \eqref{eq:intersection}, $\OO_D$ is a relatively open subset of $Q_\alg$.

\begin{Def}
A function $F=F_1+F_2\ui:D \lra \alg_\CC$ is called \emph{stem function} if $F(\zbar) = \overline{F(z)}$ for every $z \in D$. The stem function $F=F_1+F_2\ui$ on $D$ induces a \emph{left slice function} $\II(F):\OO_D \lra \alg$ on $\OO_D$ as follows. Let 
$x \in \OO_D$. By \eqref{cone = union of planes}, there exist $r, s \in \RR$ and $J \in \su_\alg$ such that $x = r+sJ$. Then we set:
\begin{equation}\label{I(F)=}
  \II(F)(x):=F_1(z)+J \, F_2(z), \quad \text{where $z=r+s i \in D$}.
\end{equation}
\end{Def}

The reader observes that the definition of $\II(F)$ is well-posed. In fact, if $x \in \OO_D \cap \RR$, then $r = x$, $s = 0$ and 
$J$ can be arbitrarily chosen in $\su_\alg$. However, $F_2(z)=0$ and hence $\II(F)(x)=F_1(x)$, independently from the choice of 
$J$. If $x \in \OO_\alg \smallsetminus \RR$, then $x$ has the following two expressions: $x = r+sJ = r+(-s)(-J)$, where $r=\re(x)$, 
$s=\sqrt {n(\im(x))}$ and $J=s^{-1}\im(x)$. Anyway, if $z := r+s i$, we have:
$\II(F)(r+(-s)(-J))=F_1(\overline{z})+(-J)F_2(\overline{z})=F_1(z)+(-J)(-F_2(z))=F_1(z)+J \, F_2(z)=\II(F)(r+sJ).$

It is important to observe that every left slice function $f:\OO_D \lra \alg$ is induced by a unique stem function $F=F_1+F_2\ui$. In fact, it is easy to verify that, if $x_J = r+s J \in \OO_D$ and $z=r+s i \in D$, then $F_1(z)=(f(x_J)+f(x_J^c))/2$ and
$F_2(z)=-J \, (f(x_J)-f(x_J^c))/2$. Therefore
one obtains the following representation formula (see \cite[Proposition~5]{GhiPer11}):
\begin{equation}\label{repr}
f(x)=\frac12(f(x_J)+f(x_J^c))-\frac I2(J \, (f(x_J)-f(x_J^c)))\quad \forall x=r+sI\in\Omega_D,
\end{equation}
which implies the next  proposition.

\begin{Prop}\label{P:representation}
Every left slice function is uniquely determined by its values on a plane $\ci_J$.
\end{Prop}

Let us introduce a relevant subclass of left slice functions. 

\begin{Def} \label{eq:real-slice}
Let $F = F_1 + F_2\ui : D \lra \alg_\CC$ be a stem function on $D$. The left slice function $f=\II(F)$ induced by $F$ is said to be \emph{slice preserving} if $F_1$ and $F_2$ are real-valued.
\end{Def}

One can prove the following (cf. \cite[Proposition 10]{GhiPer11}).

\begin{Prop}\label{condition for f real}
Let $Q_\alg\ne\ci$. A left slice function $f$ is slice preserving if and only if $f(\OO_D \cap \CC_J) \subseteq \CC_J$ for every $J \in \su_\alg$.
\end{Prop}

In general, the pointwise product of slice functions is not a slice function. However, if $F=F_1+F_2\ui$ and $G=G_1+G_2\ui$ are stem functions, then it is immediate to see that their pointwise product
\begin{equation}\label{pointwise product of stems}
FG=(F_1G_1-F_2G_2)+(F_1G_2+F_2G_1)\ui   \notag
\end{equation}
is again a stem function. In this way, we give the following definition.

\begin{Def}\label{D:slice product}
Let $f=\II(F)$ and $g=\II(G)$ be two left slice functions on $\OO_D$. We define the \emph{slice product $f \cdot g$} as the left slice function $\II(FG)$ on $\OO_D$.
\end{Def}

In the remainder of the paper, a constant function will be denoted by its value: if $f(x) = a \in \alg$ for every $x \in \Omega_D$, we will write $f = a$.
Observe that it holds $a\cdot b=ab$ for every pair of constants $a,b\in\alg$.


\subsection{Slice regular functions}

Our next aim is to recall the concept of left slice regular function, which generalizes the notion of holomorphic function from $\CC$ to any real alternative *-algebra like $\alg$.

Let $F : D \funzione \alg_\CC$ be a stem function with components $F_1, F_2 : D \funzione \alg$.  Since we endow $\alg$ with the topology induced by any norm on it as a finite dimensional real vector space, if $z = r+ s i$, $r, s \in \erre$, denotes the complex variable in 
$\ci$, it makes sense to consider the partial derivatives $\partial F/\partial r$, $\partial F/\partial s\,\ui$, which are also stem functions. 

\begin{Def}\label{def:slice-regular}
Let $F = F_1 + F_2\ui :  D  \funzione \alg_\CC$ be a stem function belonging to $C^1(D;\alg_\ci)$ (i.e. 
$F_1, F_2 \in C^1(D;\alg)$). Let us denote by $z =r + s i$, $r, s \in \erre$, the complex variable in $\ci$. We define the continuous stem functions  $\partial F/\partial z : D \funzione \alg_\ci$ and $\partial F/\partial \overline{z} : D \funzione \alg_\ci$ by
\begin{equation}
  \frac{\partial F}{\partial z} := \frac{1}{2}\left(\frac{\partial F}{\partial r} - \frac{\partial F}{\partial s}\ui\right), \qquad
  \frac{\partial F}{\partial \overline{z}} := \frac{1}{2}\left(\frac{\partial F}{\partial r} + \frac{\partial F}{\partial s}\ui\right).  \notag
\end{equation}
If $f = \II(F) : \OO_D \funzione \alg$ is the left slice function induced by $F$, we define the continuous slice functions
\begin{equation}
  \frac{\partial f}{\partial x} := \II \left(\frac{\partial F}{\partial z}\right), \qquad
 \frac{\partial f}{\partial x^c} := \II \left(\frac{\partial F}{\partial \overline{z}}\right).  \notag
\end{equation}
We say that $f=\II(F)$ is \emph{left slice regular} if $\partial f/\partial x^c = 0$.
\end{Def} 

%

Significant examples of slice regular functions are the \emph{polynomials with right coefficients in $\alg$}, i.e. functions 
$p : Q_\alg \funzione \alg$ of the form $p(x) = \sum_{k=0}^n x^k c_k$ with $c_k \in \alg$, $n \in \enne$. 
If we identify, for simplicity, the imaginary unit $\mathbf{i}$ of $\mathbb{A}_{\mathbb{C}}=\mathbb{A} \otimes_{\mathbb{R}}
\mathbb{C}$ with the imaginary unit $i$ of $\mathbb{C}=\mathbb{R} \otimes_{\mathbb{R}} \mathbb{C}$,
we have that $p = \II(P)$ where $P : \ci \funzione \alg_\ci$ is defined by $P(z) = \sum_{k=0}^n z^k c_k$. Given two polynomials $p(x) := \sum_{k = 0}^{n} x^k c_k$ and $q(x) := \sum_{k = 0}^{m} x^k d_k$, their \emph{star product} 
$p * q : Q_\alg \funzione \alg$ is defined by setting 
\begin{equation}
  (p*q)(x) := \sum_{j=0}^{n+m} x^j\bigg(\sum_{k+h=j} c_k d_h\bigg),  \notag
\end{equation}
i.e. we impose the commutativity for the product of the variable $x$ with the coefficients. Note that $p*q \neq pq$, the pointwise product. In fact the following {result} holds (\cite[Proposition 12]{GhiPer11}).

\begin{Prop}\label{slice prod. polinomi}
If $p$ and $q$ are polynomials with right coefficients in $\alg$, then $  p*q = p \cdot q$,
i.e.\ the star product is equal to the slice product.
\end{Prop}

For any $y \in Q_\alg$, the \emph{characteristic polynomial of $y$} is the slice preserving slice regular function $\Delta_y : Q_\alg \funzione Q_\alg$ induced by the stem function $\CC \funzione \alg_\CC : z 	\mapsto (y-z)(y^c-z)=z^2-zt(y)+n(y)$,
\begin{equation}\label{delta}
\Delta_y(x) :=(y-x) \cdot_x (y^c-x)=x^2-xt(y) + n(y), \qquad x \in Q_\alg.  \notag
\end{equation}
Observe that $y \in \ci_J$ for some $J \in \su_\alg$, therefore the set of zeroes of $\Delta_y$ is 
\begin{equation}
  \su_y:=\{\xi+\eta K \in Q_\alg \, : \, \xi,\eta\in\erre,   K \in \su_\alg,  y= \xi + \eta J\}.  \notag
\end{equation}
Hence we can define the  left slice regular function $C_y : Q_\alg \smallsetminus \su_y \funzione \alg$ by setting
\begin{equation}\label{cauchy kernel}
  C_y(x) := \II((z^2-zt(y)+n(y))^{-1}(y^c-z)) =\Delta_y(x)^{-1}(y^c-x).
\end{equation}
We say that $C_y$ is the \emph{Cauchy kernel for left slice regular functions on $\alg$}.

By definition \eqref{cauchy kernel}, $C_y$ turns out to be the inverse of the slice regular function $y-x$ with respect to the slice product; that is, $C_y(x) \cdot_x (y-x)=(y-x) \cdot_x C_y(x)=1$. If $x \in \CC_J$, then $x$ commutes with $y$ and $y^c$, thus $\Delta_y(x)=(y^c-x)(y-x)$ and, if $x \not\in \{y,y^c\}$, $\Delta_y(x)^{-1}=(y-x)^{-1}(y^c-x)^{-1}$. This ensures that 
\begin{equation}\label{cauchy kernel in the Jplane}
C_y(x) = (y-x)^{-1} \qquad \forall x,y \in \mathbb{C}_J,\ x \neq y, \ x \neq y^c.  \notag
\end{equation}


\section{Cauchy integral formula} \label{Cauchy}

We start with a lemma where we introduce a useful complex structure on $\alg$.

\begin{Lem}\label{L:phiJ C-isomorphism}
If $J \in \su_\alg$, then the mapping $\phi_J : \ci \funzione \ci_J$ defined by
\begin{equation}\label{phiJ}
 \phi_J(r+si) := r+sJ, \qquad r,s \in \erre,  
\end{equation}
is a complex algebra isomorphism. 
Moreover, the product $\ci \times \alg \funzione \alg: (z,x) \mapsto zx$ defined by
\begin{equation}\label{product in V_J}
  z x := \phi_J(z) x, \qquad z \in \ci, \ x \in \alg,
\end{equation}
makes $\alg$ a complex vector space.
\end{Lem}

\begin{proof}
The fact that $\phi_J$ is an isomorphism is an easy consequence of  \eqref{r(ab) = (ra)b = a(rb)} and \eqref{ra = ar}. In order to prove that \eqref{product in V_J} makes $\alg$ a complex vector space, we need to invoke Artin's Theorem \ref{artin}: indeed if $x \in \alg$,
then $J(Jx) = (JJ)x = -x$, and this implies, together with a straightforward calculation, that $z_1(z_2x) = (z_1z_2)x$ for every $z_1, z_2 \in \ci$. The remaining axioms are trivially satisfied.
\end{proof}

From Lemma \ref{L:phiJ C-isomorphism} it follows that $\phi_J(z^{-1}) = (\phi_J(z))^{-1}$ for any $z \neq 0$, and the product of $\alg$ is commutative and associative in $\ci_J$.

It will be useful to consider $\alg_\ci$ as a complex vector space as well, indeed one can easily infer the following {lemma}.

\begin{Lem}
The product 
$\ci \times \alg_\ci \funzione \alg_\ci: (z,v) \mapsto zv$:
\begin{equation}\label{product by a complex scalar in A_C}
  (r+si)(x+y\ui) := (rx - sy) + (ry + sx)\ui
\end{equation}
for $z = r + si$, $v = x+y\ui$, $r, s \in \erre$, $x, y \in \alg$,
makes $\alg_\ci$ a complex vector space.
\end{Lem}

Therefore we have the following

\begin{Lem}\label{L:Phi continuous homo}
If $J \in \su_\alg$, then the mapping $\Psi_J : \alg_\ci \funzione \alg$ defined by
\begin{equation}\label{PsiJ}
  \Psi_J(a+b\ui) := a + Jb, \qquad a,b \in \alg,
\end{equation}
is a continuous complex vector space {linear map} when $\alg_\ci$ and $\alg$ are endowed with the complex structures defined by \eqref{product by a complex scalar in A_C} and \eqref{product in V_J}{, respectively}.
\end{Lem}

\begin{proof}
Let $z = r + si \in \ci$, $r, s \in \erre$, and $v = x + y\ui \in \alg_\ci$, $x, y \in \alg$. Recalling definitions 
\eqref{product by a complex scalar in A_C} and \eqref{product in V_J}, and using \eqref{r(ab) = (ra)b = a(rb)}, \eqref{ra = ar}, and Artin's Theorem \ref{artin}, we get 
\begin{align}  
  z \Psi_J(v) 
    & = \phi_J(z)(x + Jy) \notag \\    
    & = (r+sJ)(x + Jy) \notag \\
    & =  rx+ r(Jy) + (sJ)x + (sJ)(Jy) \notag \\
    & =  rx+ J(ry+sx) + s((JJ)y) \notag \\
    & = (rx-sy) + J(ry+sx) \notag \\  
    & = \Psi_J((rx-sy) + (ry+sx)\ui) \notag \\
    & =  \Psi_J((r+si)(x+y\ui)) = \Psi_J(zv). \notag
\end{align}
Thus $\Psi_J$ is homogeneous. The additivity and the continuity are clear.
\end{proof}

Bearing in mind Artin's Theorem 2.1, it is easy to prove the following {lemma.}

\begin{Lem}\label{cond for f.g=fg}
Let $J\in \su_\alg$. Let 
$f = \II(F)$, $g = \II(G)$, and $h = \II(H)$ be slice functions on $\Omega_D$. 
If $F$ takes values in $\ci_J\otimes_\erre\ci$, then $f \cdot g = f g$ on $\Omega_D\cap\ci_J$.
If $F,G,H$ take values in $\ci_J\otimes_\erre\ci$,  then $(f\cdot g)\cdot h = f \cdot (g\cdot h)$ and $f\cdot g=g\cdot f$.
\end{Lem}

In the situation of previous Lemma \ref{cond for f.g=fg}, we will omit the parentheses.

We are now in position to prove the general Cauchy integral representation formula for slice functions, the natural noncommutative
and nonassociative generalization of the classical complex Cauchy integral formula:
\begin{equation}\label{classical cauchy formula}
  F(z) = \frac{1}{2\pi i} \int_{\partial D} \frac{F(\zeta)}{\zeta - z} \de \zeta - 
             \frac{1}{2\pi i}\int_D  \frac{(\partial F/\partial \overline{\zeta})(\zeta)}{\zeta - z} \de \zetabar \wedge \de \zeta
\end{equation}
holding for $F \in C^1(\overline{D};\ci)$, where $D \subseteq \ci$ is a bounded domain with piecewise $C^1$ boundary. Here $\frac{1}{2i}\de \zetabar \wedge \de \zeta$ is the 2-dimensional Lebesgue measure on $\ci$. 

\begin{Thm}[Cauchy integral formula]\label{T:Cauchy formula}
Let $D \subseteq \ci$ be a bounded domain, $J \in \su_\alg$ and $D_J := \Omega_D \cap \ci_J$. Let $\partial D_J$ denote the boundary of $D_J$ in $\ci_J$ and assume that it is piecewise $C^1$. If $f = \II(F) : \Omega_D \funzione \alg$ is a left slice function and $F \in C^1(\overline{D};\alg_\ci)$, then
\begin{equation}\label{Cauchy formula}
  f(x) = \frac{1}{2\pi} \int_{\partial D_J} \left[C_y \cdot \left(J^{-1}\de y\ f(y)\right)\right](x) - 
           \frac{1}{2\pi} \int_{D_J} \Big[ C_y \cdot \Big( J^{-1} \de y^c \wedge \de y\ \frac{\partial f}{\partial y^c}(y) \Big) \Big](x)  \notag
\end{equation}
for every $x \in \Omega_D$.
\end{Thm}

Before showing the proof, some remarks should be made. 

\begin{Rem}\label{remark31}\ 
\begin{itemize}
\item[(i)]
As we mentioned in the {i}ntroduction, the position of the ``differentials'' inside integrals of $\alg$-valued functions is important, so a rigorous definition is in order. We limit ourselves to the integrals involved in the Cauchy formula. If $a, b \in \erre$, $a < b$, and 
$\alpha : \clint{a,b} \funzione \ci_J$ is a piecewise $C^1$ parametrization of the (counterclockwise oriented) Jordan curve 
$\partial D_J$ in the plane $\ci_J$, then
\[
  \int_{\partial D_J} \left[ C_y \cdot \left( J^{-1}\de y\ f(y) \right) \right](x) :=
  \int_a^b \left[ C_{\alpha(t)} \cdot \left( J^{-1} \alpha'(t) \ f(\alpha(t)) \right) \right](x) \de t,
\]
$\alpha'$ being the derivative of $\alpha$. The second integral is simply
\[
  \int_{D_J} \Big[ C_y \cdot \Big( J^{-1} \de y^c \wedge \de y\ \frac{\partial f}{\partial y^c}(y) \Big) \Big](x) := 
 2 \int_{D_J} \Big[ C_{\rho+\sigma J} \cdot   \frac{\partial f}{\partial y^c}(\rho+\sigma J)  \Big](x) \de \rho \de \sigma,
\]
$(\rho,\sigma)$ being the coordinates of $y$ in $\ci_J = \{y = \rho+\sigma J\ :\ \rho, \sigma \in \erre\}$. Hence 
$\frac{J^{-1}}{2}dy^c \wedge dy$ may be considered as the 2-dimensional Lebesgue measure on $\ci_J \simeq \erre^2$.
\item[(ii)]
In the two integrand functions, the slice product $\cdot$ is computed with respect to the variable $x$:  $J^{-1}\de y\ f(y)$ and
$J^{-1} \de y^c \wedge \de y\ (\partial f/\partial y^c)(y)$ are here constant functions (w.r.t. $x$), $y$ being the (fixed) integration variable.
Using the notation $\cdot_x$ for the slice product w.r.t. $x$, the Cauchy formula can be written in the following way:
\begin{equation}\label{Cauchy formula-2 form}
  f(x) 
     = \frac{1}{2\pi} \int_{\partial D_J} C_y(x) \cdot_x \left(J^{-1}\de y\ f(y)\right) - 
           \frac{1}{2\pi} \int_{D_J} C_y(x) \cdot_x \Big( J^{-1} \de y^c \wedge \de y\ \frac{\partial f}{\partial y^c}(y) \Big). \notag
\end{equation}
\item[(iii)]
There are no parentheses in the term $J^{-1}\de y\ f(y) = J^{-1} \alpha'(t) \ f(\alpha(t)) dt$, because it belongs to the subalgebra generated by $J$ and $f(y)$, thus, by Artin's Theorem \ref{artin}, this product is associative.
\item[(iv)]
The formula of Theorem~\ref{T:Cauchy formula} reduces in the associative case to what stated in  \cite[Theorem~27]{GhiPer11}.
\end{itemize}
\end{Rem}

\begin{proof}[Proof of Theorem \ref{T:Cauchy formula}]
Let us first prove the theorem under the assumption
\begin{equation}\label{phiJ(z)=x}
    x \in D_J.
\end{equation}
Let $z := \phi_J^{-1}(x)\in\ci$ (see \eqref{phiJ}). Observe that, from \eqref{PsiJ} and \eqref{I(F)=}, we get
\begin{equation}\label{f(phiJ)=Phi_J(F)}
  f(\phi_J(w)) = \Psi_J(F(w)) \qquad \forall w \in \overline{D},
\end{equation}
\begin{equation}\label{f(partial phiJ)=Phi_J(partial F)}
  \frac{\partial f}{\partial y^c}(\phi_J(w)) = \Psi_J\Big(\frac{\partial F}{\partial \overline{w}}(w)\Big) 
  \qquad \forall w \in D.
\end{equation}
Let $\gamma : [0,1] \funzione \ci$ be a Jordan curve whose trace is $\partial D$ (counterclockwise oriented) and let $(\rho,\sigma)$ denote the real coordinates of $\zeta = \rho + \sigma i \in \ci$. Since 
 $F \in C^1(\overline{D};\alg_\ci)$, we can apply the classical vector complex Cauchy formula, which can be easily deduced from \eqref{classical cauchy formula} by means of the Hahn-Banach theorem, or simply using coordinates. We get
\begin{align}
  F(z) 
    & = \frac{1}{2\pi i} \int_{\partial D} \frac{F(\zeta)}{(\zeta - z)} \de \zeta - 
           \frac{1}{2\pi i} \int_D \frac{(\partial F/\partial \overline{\zeta})(\zeta)}{(\zeta-z)} \de \zetabar \wedge \de \zeta \notag \\
    & = \frac{1}{2\pi i} \int_0^1 \gamma'(t)\frac{F(\gamma(t))}{\gamma(t)-z}  \de t -
           \frac{1}{\pi} \int_D \frac{(\partial F/\partial \overline{\zeta})(\zeta)}{(\zeta-z)} \de \rho \de \sigma, \notag
\end{align}
where the product by a complex scalar in the integrand functions is defined by \eqref{product by a complex scalar in A_C}. Thus, recalling from Lemma \ref{L:Phi continuous homo} that $\Psi_J : \alg_\ci \funzione \alg$ is $\ci$-linear and continuous when $\alg_\ci$ and $\alg$ are endowed with the complex vector structures defined by \eqref{product by a complex scalar in A_C} and
\eqref{product in V_J}, we get
\begin{align}
  \Psi_J(F(z)) 
    & = \frac{1}{2\pi} \int_0^1 \Psi_J\left( \frac{\gamma'(t)}{(\gamma(t)-z)i} F(\gamma(t)) \right) \de t -
    \frac{1}{\pi} \int_D \Psi_J\left( \frac{(\partial F/\partial \overline{\zeta})(\zeta)}{(\zeta-z)} \right) \de \rho \de \sigma \notag \\
    & = \frac{1}{2\pi} \int_0^1 \frac{\gamma'(t)}{(\gamma(t)-z)i} \Psi_J(F(\gamma(t))) \de t -
           \frac{1}{\pi} \int_D \frac{1}{\zeta-z} \Psi_J \Big( \frac{\partial F}{\partial \overline{\zeta}}(\zeta)\Big) \de \rho \de \sigma \notag \\
     & = \frac{1}{2\pi} \int_0^1 \phi_J\left( \frac{\gamma'(t)}{(\gamma(t)-z)i} \right) \Psi_J(F(\gamma(t))) \de t  \\&-
            \frac{1}{\pi} \int_D \phi_J\left( \frac{1}{\zeta-z} \right) \Psi_J\Big( \frac{\partial F}{\partial \overline{\zeta}}(\zeta)\Big) 
              \de \rho \de \sigma. \notag
\end{align}
Now observe that $\gamma_J := \phi_J \circ \gamma : [0,1] \funzione \ci_J$ is a {(counterclockwise oriented)} parametrization of $\partial D_J$ and that
$\gamma_J' = \phi_J \circ \gamma'$. Hence, using \eqref{phiJ(z)=x}--\eqref{f(partial phiJ)=Phi_J(partial F)}, Lemma	\ref{L:phiJ C-isomorphism}, 
\eqref{cauchy kernel in the Jplane}, and Artin's Theorem \ref{artin}, we deduce that
\begin{align}
 f(x) & = \Psi_J(F(z)) \notag \\
  & = \frac{1}{2\pi} \int_0^1 \left((\gamma_J(t)-x)^{-1} J^{-1} \gamma_J'(t)\right) f(\gamma_J(t)) \de t -
        \frac{1}{\pi} \int_{D_J} (y-x)^{-1} \frac{\partial f}{\partial y^c}(y) \de  \rho \de \sigma \notag \\
  & = \frac{1}{2\pi} \int_0^1 \left(C_{\gamma_J(t)}(x) J^{-1} \gamma_J'(t)\right) f(\gamma_J(t)) \de t -
        \frac{1}{\pi} \int_{D_J} C_y(x) \frac{\partial f}{\partial y^c}(y) \de  \rho \de \sigma \notag \\
  & = \frac{1}{2\pi} \int_0^1 C_{\gamma_J(t)}(x) \left(J^{-1} \gamma_J'(t) f(\gamma_J(t))\right) \de t -
        \frac{1}{\pi} \int_{D_J} C_y(x) \frac{\partial f}{\partial y^c}(y) \de  \rho \de \sigma. \label{PhiJ(F(z))=....-2}
\end{align}
Now let us observe that if $a \in \alg$ and $y \in D_J$, then, thanks to Lemma~\ref{cond for f.g=fg}, it follows that
\[
  (C_y \cdot a)(x) = C_{y}(x) a \qquad \forall x \in \ci_J, 
\]
where $a$ denotes the constant function taking the value $a$. Therefore from \eqref{PhiJ(F(z))=....-2} we get
\begin{align}               
  f(x) & = \frac{1}{2\pi} \int_0^1 \left[ C_{\gamma_J(t)} \cdot \left(J^{-1} \gamma_J'(t)\ f(\gamma_J(t)) \right) \right](x) \de t -
         \frac{1}{\pi} \int_{D_J} \Big( C_y \cdot  \frac{\partial f}{\partial y^c}(y) \Big)(x) \de \rho \de \sigma \notag \\
  & = \frac{1}{2\pi} \int_{\partial D_J} \left[C_y \cdot \left(J^{-1} \de y\ f(y) \right) \right](x) - 
         \frac{1}{2 \pi} \int_{D_J} \Big[ C_y \cdot \Big( J^{-1} \de y^c \wedge \de y\ \frac{\partial f}{\partial y^c}(y) \Big) \Big](x), \notag
\end{align}
which proves the theorem in the case $x \in D_J$. In order to conclude, it is enough to invoke Proposition \ref{P:representation}, since $f$ and the function on the right hand side of the previous formula are slice functions on $\Omega_D$.
\end{proof}
  
\begin{Cor}[Cauchy formula for slice regular functions]\label{Cauchy formula regular}
Under the same assumptions of Theorem \ref{T:Cauchy formula}, if $f$ is left slice regular on $\Omega_D$, 
then 
\begin{equation}\label{Cauchy formula regular-eq}
  f(x) 
     = \frac{1}{2\pi} \int_{\partial D_J} \left[ C_y \cdot \left(J^{-1}\de y\ f(y)\right) \right](x) 
     = \frac{1}{2\pi} \int_{\partial D_J} C_y(x) \cdot_x \left(J^{-1}\de y\ f(y)\right) 
\end{equation}
for each $x\in\Omega_D$.
\end{Cor}

\section{Applications: series expansions}

The new Cauchy formula \eqref{Cauchy formula regular-eq} permits to prove series expansions of slice regular functions following the lines of the classical method used in the case of holomorphic functions of a complex variable: the expansion of the Cauchy kernel. 

We assume that $\alg$ is equipped with a norm $\| \cdot \|_\alg$ satisfying the property: $\|x\|_\alg=\sqrt{n(x)}$ for each $x \in Q_\alg$ (e.g.\ the Euclidean norm on $\quat$, $\oct$ and  $\erre_{0,q}$). As shown in \cite[\S3.1]{GhiPer14}, there exists $C_1>0$ such that $\|xy\|_\alg \leq C_1\, \|x\|_\alg\|y\|_\alg$ if $x,y\in \alg.$
We recall the definition of the metric $\sigma_\alg$ on $Q_\alg$ (cf.\ \cite{GenSto12} and \cite{GhiPer14}). For $x_0\in\ci_J$ and $x\in Q_\alg$, 
\begin{equation*}
\sigma_\alg(x,x_0):=
\begin{cases}
 \|x-x_0\|_\alg &\text{\ if }x\in \ci_J\\
\sqrt{{|\re(x)-\re(x_0)|}^2+{\left(\|\im(x)\|_\alg+\|\im(x_0)\|_\alg\right)}^2} &\text{\ if }x\notin \ci_J.
\end{cases}
\end{equation*}
Given $r \in \erre^+$, let $\Sigma_\alg(x_0,r)=\{x\in Q_\alg:\sigma_\alg(x,x_0)<r\}$ be the $\sigma_\alg$--ball of $Q_\alg$ centered at $x_0$ of radius $r$.  

Let $D$ be as in Theorem \ref{T:Cauchy formula} and $x_0$ a fixed point in $D_J$. For $x\in D_J$, $y\in\partial D_J$ and $n\in\enne$, the Cauchy kernel expands as:
\[C_y(x)=(y-x)^{-1}=\sum_{k=0}^n (x-x_0)^k(y-x_0)^{-k-1}+(x-x_0)^{n+1}(y-x)^{-1}(y-x_0)^{-n-1}.
\]
Let $(x-x_0)^{\cdot n}$ denote the $n$-th power of the slice function $x-x_0$ w.r.t.\ the slice product. Thanks to Proposition~\ref{P:representation} and Lemma~\ref{cond for f.g=fg}, we get that 
\[C_y(x)=\sum_{k=0}^n (x-x_0)^{\cdot k}\cdot_x(y-x_0)^{-k-1}+(x-x_0)^{\cdot n+1}\cdot_x C_y(x)\cdot_x(y-x_0)^{-n-1}
\]
for each $x\in\Omega_D$, $y\in\partial D_J$ and $n\in\enne$.
 Applying the Cauchy formula (Corollary~\ref{Cauchy formula regular}) to a slice regular function $f$, we obtain, for every $x\in\Omega_D$,
\[
  f(x) =  \sum_{k=0}^n (x-x_0)^{\cdot k}\cdot_x \frac{1}{2\pi} \int_{\partial D_J}(y-x_0)^{-k-1}  \left(J^{-1}\de y\ f(y)\right) + R_n(x)
 \notag
\]
with \[R_n(x)=(x-x_0)^{\cdot n+1}\cdot_x\frac{1}{2\pi} \int_{\partial D_J} C_y(x) \cdot_x (y-x_0)^{-n-1} \left(J^{-1}\de y\ f(y)\right).\]
Now  choose $r'>r>0$ such that $\Sigma_\alg(x_0,r')\subseteq\Omega_D$. As shown in \cite{GhiPer14}, $\Sigma_\alg(x_0,r')\cap\ci_J$ is an open disk $B_J(x_0,r')$ in $\ci_J$ and $\Sigma_\alg(x_0,r)=B_J(x_0,r)\cup \Omega(x_0,r)$, with $\Omega(x_0,r)$ a circular set.
From the representation formula for slice functions \eqref{repr} applied on $\Omega(x_0,r)$, one obtains 
that there exists a constant $C_2>0$ such that
\[
\|R_n(x)\|_\alg\le C_2 \sup_{x'\in B_J(x_0,r)}\|R_n(x')\|_\alg\quad\forall x\in\Sigma_\alg(x_0,r), \forall n\in\enne.
\]
Using Lemma~\ref{cond for f.g=fg}, we get that for every $x'\in B_J(x_0,r)$,
\[
R_n(x')=(x'-x_0)^{n+1}\frac{1}{2\pi} \int_{\partial D_J} C_y(x')  (y-x_0)^{-n-1} \left(J^{-1}\de y\ f(y)\right).
\]
Since $\|x'-x_0\|_\alg<r$ and $\|y-x_0\|_\alg\ge r'$ for each $x'\in B_J(x_0,r)$ and $y\in\partial D_J$, there exists $C_3>0$ 
 such that
\[
\|R_n(x)\|_\alg\le C_3 \left(\frac r{r'}\right)^{n+1}\quad\forall x\in\Sigma_\alg(x_0,r),\ \forall n\in\enne.
\]
We have proved the following result:

\begin{Prop}\label{prop4_1}
Let $D \subseteq \ci$ be a bounded domain, with $\partial D$ piecewise $C^1$. Let $r\in\erre^+$ with   $\overline{\Sigma_\alg(x_0,r)}\subseteq\Omega_D$.
If $f$ is left slice regular on $\Omega_D$ and of class $C^1$ on $\overline{\Omega}_D$, then there exists a unique sequence $(a_k)_{k\in\enne}$ in $\alg$, defined, for each $J\in\su_\alg$, by the formula
\[
a_k=\frac{1}{2\pi} \int_{\partial D_J}(y-x_0)^{-k-1}  \left(J^{-1}\de y\ f(y)\right),
\] 
such that
\begin{equation}
  f(x) = \sum_{k=0}^{+\infty}(x-x_0)^{\cdot k}\cdot_x a_k \notag
\end{equation}
with uniform convergence on $\Sigma_\alg(x_0,r)$. 
\end{Prop}

A drawback of using $\sigma_\alg$ is that this metric is finer than the Euclidean one. This problem was solved in \cite{StoppatoAdvMath2012} introducing a new pseudo--metric and a different series expansion.
 For $x,x_0\in Q_\alg$, let us define
\[
\sto_\alg(x,x_0):=\sqrt{\|\Delta_{x_0}(x)\|_\alg}.
\]
The function $\sto_A$ is a pseudo--metric on $Q_\alg$  (cf.~\cite{StoppatoAdvMath2012} and \cite{GhiPer14}), called \emph{Cassini pseudo--metric}, whose induced topology is strictly coarser than the Euclidean one. 
Given $r\in\erre^+$,  the $\sto_\alg$--ball  $\Sto_\alg(x_0,r)=\{x\in Q_\alg:\sto_\alg(x,x_0)<r\}$ centered at $x_0$ of radius $r$ is a circular set.
For each $m\in\enne$ consider the slice regular functions, called \emph{spherical polynomials},
\[
\stx_{x_0,2m}(x):=\Delta_{x_0}(x)^m,
\quad
\stx_{x_0,2m+1}(x):=\Delta_{x_0}(x)^m(x-x_0).
\]
Let $D$ be as above and let $x_0$ be a fixed point in $D_J$. Using alternatively the two equalities
\begin{align*} 
C_y(x)&=(y-x)^{-1}=(y-x_0)^{-1}+(x-x_0)(y-x)^{-1}(y-x_0)^{-1},\\
C_y(x)&=(y-x)^{-1}=(y-x_0^c)^{-1}+(x-x_0^c)(y-x)^{-1}(y-x_0^c)^{-1}
\end{align*}
for $x\in D_J$ and $y\in\partial D_J$,  and then applying Proposition~\ref{P:representation} and Lemma~\ref{cond for f.g=fg}, we obtain, for each $x\in\Omega_D$ and $n\in\enne$:
\[\label{eq:cauchy-kernel}
C_y(x)=\sum_{k=0}^n \stx_{x_0,k}(x)\cdot_x\stx_{x_0,k+1}(y)^{-1}+\stx_{x_0,n+1}(x)\cdot_x C_y(x)\cdot_x\stx_{x_0,n+1}(y)^{-1}.
\]
We now proceed as above.
From Corollary~\ref{Cauchy formula regular} we get, for every $x\in\Omega_D$,
\[
  f(x) =  \sum_{k=0}^n  \stx_{x_0,k}(x)\cdot_x \frac{1}{2\pi} \int_{\partial D_J}\stx_{x_0,k+1}(y)^{-1}  \left(J^{-1}\de y\ f(y)\right) + R'_n(x)
 \notag
\]
with \[R'_n(x)= \stx_{x_0,n+1}(x)\cdot_x\frac{1}{2\pi} \int_{\partial D_J}C_y(x)\cdot_x\stx_{x_0,n+1}(y)^{-1}\left(J^{-1}\de y\ f(y)\right).\]
Let $r'>r>0$ such that $\Sto_\alg(x_0,r')\subseteq\Omega_D$. The set $\Sto_J(x_0,r):=\Sto_\alg(x_0,r)\cap\ci_J$ is an open subset of $\ci_J$ bounded by a Cassini oval.
From the representation formula \eqref{repr} we get
that there exists $C_3>0$ such that
\[
\|R'_n(x)\|_\alg\le C_3 \sup_{x'\in \Sto_J(x_0,r)}\|R'_n(x')\|_\alg\quad\forall x\in\Sto_\alg(x_0,r), \forall n\in\enne.
\]
Using Lemma~\ref{cond for f.g=fg}, we get that for every $x'\in \Sto_J(x_0,r)$,
\[
R'_n(x')= \stx_{x_0,n+1}(x')\frac{1}{2\pi} \int_{\partial D_J}C_y(x')\stx_{x_0,n+1}(y)^{-1}\left(J^{-1}\de y\ f(y)\right).
\]
Let  $x'\in \Sto_J(x_0,r)$ and  $y\in\partial D_J$. Since $\sto_\alg(x',x_0)<r$ and $\sto_\alg(y,x_0)>r'$, from \cite[Lemma~5.2]{GhiPer14} we get
\begin{align*}
\|&\stx_{x_0,n+1}(x')\|\le r^n(r+2\|\im(x_0)\|),
\\
\|&\stx_{x_0,n+1}(y)\|\ge (r')^{n+2}(r'+2\|\im(x_0)\|)^{-1}.
\end{align*}
Therefore there exists $C_4>0$ 
 such that 
\[
\|R'_n(x)\|_\alg\le C_4 \left(\frac r{r'}\right)^{n}\quad\forall x\in\Sto_\alg(x_0,r),\ \forall n\in\enne,
\]
from which we get the so-called \emph{spherical expansion} of $f$:
\begin{Prop}\label{prop4_2}
Let $D \subseteq \ci$ be a bounded domain, with $\partial D$ piecewise $C^1$. Let $r\in\erre^+$ with   $\overline{\Sto_\alg(x_0,r)}\subseteq\Omega_D$.
If $f$ is left slice regular on $\Omega_D$ and of class $C^1$ on $\overline{\Omega}_D$, then there exists a unique sequence $(s_k)_{k\in\enne}$ in $\alg$, defined, for each $J\in\su_\alg$, by the formula
\[s_k=\frac{1}{2\pi} \int_{\partial D_J}\stx_{x_0,k+1}(y)^{-1}  \left(J^{-1}\de y\ f(y)\right),\] 
such that
\begin{equation}
  f(x) = \sum_{k=0}^{+\infty}\stx_{x_0,k}(x)\cdot_x s_k \notag
\end{equation}
with uniform convergence on $\Sto_\alg(x_0,r)$. 
\end{Prop}

\begin{ex}[cf. \cite{GhiPer14} Ex.~5.6]
Let $A=\q$ and let $J\in\su_\q$ be fixed. Consider the slice regular function $f$ on $\q\smallsetminus\RR$ defined, for each $x=r+s I$, with $s>0$, by
\[f(x)=1-IJ.\]
We compute the power and spherical expansions of $f$ at $y=J$. 
The maximal $\sigma_\q$--ball centered at $J$ on which the power expansion converges to $f$ is the domain $\Sigma_{\q}(J,1)=\{q\in\CC_J\;|\;|q-J|<1\}$, which has empty interior w.r.t.\ the euclidean topology of $\q$. On the other hand, the spherical expansion converges to $f$ on a non-empty open domain of $\q$. 
We can compute the coefficients $a_k$ and $s_k$ by means of Propositions~\ref{prop4_1} and \ref{prop4_2}. Using the fact that $f$ assumes constant value 2 on $\ci_J^+=\{r+sJ\in\ci_J:s>0\}$ and vanishes on $\ci_J^-=\{r+sJ\in\ci_J:s<0\}$, we get
\[a_k=(\pi J)^{-1} \int_{\partial\Delta}(y-J)^{-k-1} \de y
=\begin{cases}2\text{\quad if $k=0$}\\0 \text{\quad if $k>0$}\end{cases},
\]
\[s_k=
\begin{cases}
(\pi J)^{-1} \int_{\partial\Delta}(y^2+1)^{-n-1}\de y &\text{\ if $k=2n+1$ is odd}\\
(\pi J)^{-1}\int_{\partial\Delta}(y^2+1)^{-n}(y-J)^{-1} \de y &\text{\ if $k=2n$ is even}\end{cases}\]
where $\Delta$ is a disk in $\ci_J$ centered at $J$ with radius smaller than 1.
Therefore
\[
s_0=2,\quad s_k=\begin{cases}
4^{-n}\binom{2n}n(-J)&\text{\quad if $k=2n+1$ is odd,}\\
4^{-n}\binom{2n}n&\text{\quad if $k=2n>0$ is even}\end{cases}
\]
and the spherical expansion of $f$ at $J$ takes the form
\begin{align*}
f(x)&=1-\sum_{n=0}^{+\infty} \frac1{4^n}\binom{2n}n (1+x^2)^nxJ.
\end{align*}

\end{ex}




\end{document}